\documentclass{article}
\pdfoutput=1
\usepackage{amsmath, amssymb, amsthm, mathrsfs, miyamath}
\usepackage{fourier}
\usepackage{tikz}
\usetikzlibrary{matrix,arrows,snakes}
\usepackage{main}

\title{Finiteness of Crystalline Cohomology \\ of Higher Level}
\author{Kazuaki MIYATANI}

\date{\empty}

\makeatletter
\def\thefootnote{($\ast$\@alph\c@footnote)}
\makeatother

\begin{document}

\setlength{\abovedisplayskip}{5pt}
\setlength{\belowdisplayskip}{5pt}

\maketitle

\begin{abstract}
    We prove the finiteness of crystalline cohomology of higher level.
    An important ingredient is a ``higher de Rham complex'' and
    a kind of Poincar\'e lemma for it.
\end{abstract}

\section*{Introduction.}

\subsection{Crystalline cohomology of higher level.}

Pierre Berthelot \cite{Berthelot:DMAI} generalized the notion of PD structure,
the most fundamental notion in the crystalline theory,
to that of ``$m$-PD structure (PD structure of level $m$)'' for each natural number $m$.
After replacing the classical PD structure by this notion,
we directly get a level-$m$ version of crystalline site, crystal and other crystalline concepts.
In particular, we can define $m$-crystalline cohomology over very ramified DVR
because a DVR of mixed characteristic $(0,p)$ has an $m$-PD structure on its maximal ideal if the absolute ramification index is not greater than $p^m(p-1)$.

In contrast to the simplicity of this generalization,
we cannot directly apply classical arguments to prove
fundamental properties of the $m$-crystalline cohomology such as base change and finiteness.
The main reason lies in the fact that the crystalline Poincar\'e lemma appears to be difficult
in the level-$m$ situation.
Bernard Le Stum and Adolfo Quir\'os indeed proved the so-called exact crystalline Poicar\'e lemma \cite{LeStum-Quiros:EPLCCHL},
which states that the $m$-crystalline cohomology is calculated by the ``jet complex of order $p^m$'';
unfortunately, this complex is not bounded and its local freeness is yet to be proved
(the latter point will be discussed in the next subsection).
The lack of boundedness and local freeness prevents us from proving as in the classical case
the cohomological boundedness and the base change, which we need for proving, for example, the finiteness.

In spite of this difficulty, we prove in this article the finiteness of the cohomology by using an auxiliary ``de Rham-like'' complex
and corresponding Poincar\'e lemma.
To be more precise, we first define in a local situation a complex that is,
if we ignore the differential maps, isomorphic to the usual de Rham complex;
we call this complex the ``higher de Rham complex.''
We next prove that this is a resolution of a direct sum of finitely many copies
of the structure sheaf.
Thirdly, although the arguments so far are done just locally,
we show in the global situation the cohomological boundedness and the base change theorem.
At last, by using these two properties and the exact crystalline Poincar\'e lemma,
we prove the finiteness of this cohomology.

This article is organized as follows.
Section 1 is devoted to recalling the foundation of $m$-crystalline theory
and doing some differential calculi.
In Section 2 introduced is the higher de Rham complex and the Poincar\'e lemma for it,
and in Section 3 proved is the finiteness of $m$-crystalline cohomology.
Finally, other applications of the higher de Rham complex are included in Section 4.

\subsection{A problem on local freeness of the jet complex.}
\label{ss:problem_on_freeness}

In fact, if $(S,\fa,\fb,\gamma)$ is an $m$-PD scheme (Definition \ref{def:ring}) and if $X$ is a smooth scheme over $S$,
Le Stum and Quir\'os \cite[1.4]{LeStum-Quiros:EPLCCHL} give a proof of the local freeness of each term appearing in  the jet complex of order $p^m$,
whose $r$-th term is there denoted by $\Omega_{X_m}^r$.
However, the proof for $r\geq 3$ is not correct.

Indeed, let $n$ be the dimension of $X$ over $S$, and $t_1,\dots, t_n$ a system of local coordinates.
Then, as proved in the article, there exist, for each $I\in\bN^n$ that satisfies $p^m<|I|\leq 2p^m$, 
two multi-indices $A(I), B(I)$ such that $A(I)+B(I)=I$ and that
\[
\left\{(dt)^U\otimes (dt)^V \mid (U, V) \neq (A(U+V), B(U+V))\right\}
\]
is a basis of $\Omega_{X_m}^2$.
By using this notation, the module $\Omega_{X_m}^3$ is generated by the set
\begin{align*}
& \big\{(dt)^U \otimes (dt)^V\otimes (dt)^W \mid (V, W) \neq (A(V+W), B(V+W)) \text{ and,}\\
& \qquad \text{ if } W \text{ is not } B(I) \text{ for any } I, \text{ then } (U, V) \neq (A(U+V), B(U+V)) \big\}.
\end{align*}
Their relations are given by, for each $(U, V, W)$ such that $(U, V)=(A(U+V), B(U+V))$ and that $W=B(I)$ for some $I$,
\begin{align*}
    & \sum_S \qbinom {U+V}S (dt)^{U+V-S}\otimes (dt)^S\otimes (dt)^W\\
    & \qquad  - \sum_I\qbinom{U+V}{A(I)}\qbinom{I}{A(I)}^{-1}\sum_{\substack{0<T<I\\ T\neq A(I)}}\qbinom {I}T (dt)^{U+V-A(I)}\otimes (dt)^T\otimes
    (dt)^{I-T}=0,
\end{align*}
where the first sum is taken over all $S$ such that $0<S<U+V$ and that $S$ is not equal to $A(I)$
for any $I$ with $W=B(I)$, and where the first sum in the second line moves $I$ such that $W=B(I)$;
three relatons written in that article are not correct.

Then, we have two problems with these relations.
First, it may happen that all the coefficients in this sum are non-unit.
Second, even if one of the coefficients in the first sum is a unit
(this assumption holds, for example, if there exists only one $I$ that satisfies $W=B(I)$),
$I-T$ in the second sum may again be of the form $B(I')$;
at this point, we do not know how to exclude the auxiliary generators by these relations
to prove the local freeness.

\vspace{10pt}

We refer to influences of this problem on other results in their article \cite{LeStum-Quiros:EPLCCHL}.

First, it does not affect the exact Poincar\'e lemma
at the level of $\sO_X$-modules \cite[3.3]{LeStum-Quiros:EPLCCHL},
whose proof does not use the local freeness.
Second, although the proof of the exact Poincar\'e lemma with coefficients \cite[4.7]{LeStum-Quiros:EPLCCHL} fails in general,
it remains valid if the coefficient $E$ is a flat $m$-crystal since then the assertion is reduced to the case where
$E=\sO^{(m)}_{X/S}$.
Finally, the proof of the Frobenius descent \cite[5.5]{LeStum-Quiros:EPLCCHL} fails because
it uses a proposition \cite[5.2]{LeStum-Quiros:EPLCCHL} depending on the local freeness of each term of the jet complex.
In Subsection \ref{ss:frobenius_descent}, we correct the proof of the Frobenius descent
by using our higher de Rham complex.

\subsection{Conventions.}
\label{ss:conventions}

Throughout this article, we fix a prime number $p$ and a natural number $m$
(natural number means, in this article, non-negative integer).

We assume that $p$ is nilpotent on all schemes appearing in this article.

If $k$, $k'$ and $k''$ denote natural numbers such that $k=k'+k''$, we often use the notation
\[
\binom{k}{k'}:=\frac{k!}{k'!\,k''!},\hspace{1.5em}\mbinom{k}{k'}:=\frac{q!}{q'!\,q''!}\hspace{1em}\mathrm{and}\hspace{1em} \qbinom{k}{k'}:=\binom{k}{k'}\mbinom{k}{k'}^{-1},
\]
where $q$ \resp{$q'$, $q''$} denotes the integer part of $k/p^m$ \resp{$k'/p^m$, $k''/p^m$}.
We also use the usual conventions on multi-indices; if $I=(i_1,\dots,i_n)$ and $J=(j_1,\dots,j_n)$ satisfies $J\leq I$, that is, if $j_k\leq i_k$ for all $k=1,\dots,n$, then we define
\[
\binom{I}{J}:=\prod_{k=1}^n\binom{i_k}{j_k},\hspace{1.5em}\mbinom{I}{J}:=\prod_{k=1}^n\mbinom{i_k}{j_k}
\hspace{1em}\mathrm{and}\hspace{1em}\qbinom{I}{J}:=\prod_{k=1}^n\qbinom{i_k}{j_k}.
\]
The element $(0,\dots,0,1,0,\dots,0)$ in $\bN^n$, where $1$ sits in the $i$-th entry, is denoted by $\one_i$.

\subsection{Acknowledgements.}

The author would like to express his greatest gratitude to Professor Atsushi Shiho for
introducing me to the field of crystalline theory of higher level,
reading thoroughly the draft of this paper, pointing out a lot of mistakes on it and giving him a lot of advice.

The author had a chance to talk with Professor Bernard Le Stum on this topic during his stay in Japan.
The author is grateful to him for the meaningful argument.
The author is also grateful to Professor Pierre Berthelot for his sincere and detailed answer to my questions on the crystalline site of higher level.

\section{Crystalline Site of Higher Level.}

\subsection{$m$-PD structures.}
First, let us recall some basic notions on $m$-PD structures.
The fundamental reference on this subject is Berthelot's article \cite{Berthelot:DMAI}.

\begin{definition}
    \label{def:ring}
Let $R$ be a $\bZ_{(p)}$-algebra and $\fa$ an ideal of $R$.
An {\em $m$-PD structure} on $\fa$ is a PD ideal $(\fb,\gamma)$ of $R$
that satisfies the following two conditions:

(a) $\fa^{(p^m)}+p\fa \subset \fb \subset \fa$;

(b) the PD structure $\gamma$ is compatible with the unique one on $p\bZ_{(p)}$.

Here, $\a^{(p^m)}$ denotes the ideal of $R$ generated by $x^{p^m}$ for all elements $x$ of $\a$.
We call $(\fa,\fb,\gamma)$ an {\em $m$-PD ideal} of $R$, and $(R,\fa,\fb,\gamma)$ an {\em $m$-PD ring}.

An {\em $m$-PD morphism} $(R',\fa',\fb',\gamma')\to(R,\fa,\fb,\gamma)$ between two $m$-PD rings is a PD homomorphism
$(R',\fb',\gamma')\to (R,\fb,\gamma)$ such that the image of $\fa'$ lies in $\fa$.
$\square$
\label{def:m-PD-str}
\end{definition}

\begin{definition}
Let $(R,\a,\b,\gamma)$ be an $m$-PD ring and $A$ an $R$-algebra.
We agree that $\widetilde{\b}$ denotes the ideal $\b+pR$, and that $\widetilde{\gamma}$ denotes
the PD structure on $\widetilde{\b}$ compatible with $\gamma$ and the unique one on $p\bZ_{(p)}$.

(i) We say that the $m$-PD structure $(\b,\gamma)$ {\it extends to} $A$ if the PD structure $\widetilde{\gamma}$ on $\widetilde{\b}$ extends to $A$.

(ii) Let $(I,J,\delta)$ be an $m$-PD ideal of $A$.
We say that the $m$-PD structure $(J,\delta)$ is {\em compatible with} $(\b,\gamma)$ if the following two conditions hold:

\hspace{3pt} (a) the two PD structures $\widetilde{\gamma}$ and $\delta$ are compatible;

\hspace{3pt} (b) $\widetilde{\b}A\cap I$ is a PD sub-ideal of $\widetilde{\b}A$ (equipped with the PD structure extending $\widetilde{\gamma}$).

(iii) Let $(\fa',\fb',\gamma')$ be another $m$-PD ideal of $R$. Then, this is called an {\em $m$-PD sub-ideal} of $(\fa,\fb,\gamma)$ (or simply of $\fa$)
if $\fa'$ is a sub-ideal of $\fa$ and if $(\fb',\gamma')$ is a PD sub-ideal of $(\fb,\gamma)$.
$\square$
\end{definition}

\begin{definition}
Let $(R,\a,\b,\gamma)$ be an $m$-PD ring.
For each natural number $k$ and element $x$ in $\a$, we put
\[
x^{\{k\}}:=x^r \gamma_q(x^{p^m}),
\]
where $k=p^mq+r$ and $0\leq r<p^m$. $\square$
\end{definition}

As is easily seen, this function satisfies $q!x^{\{k\}}=x^k$.
Moreover, there are relations \cite[1.3.6]{Berthelot:DMAI} similar to those of the classical divided power function;
here, we recall here a relation \cite[1.3.6 (iii)]{Berthelot:DMAI} that is used later.

\begin{proposition}\label{prop:pdbinom}
Let $(R,\fa,\fb,\gamma)$ be an $m$-PD ring. Then, for all $x, y\in\fa$ and $k\in\bN$, we have
\[
    (x+y)^{\{k\}} = \sum_{k'=0}^k\qbinom{k}{k'}x^{\{k'\}}y^{\{k-k'\}}.
\]
\end{proposition}

The following proposition, whose proof \cite[1.4.1]{Berthelot:DMAI} we do not recall here,
generalizes the notion of PD envelope.

\begin{proposition-definition}\label{thm:envelope}
Let $(R,\a,\b,\gamma)$ be an $m$-PD ring.
Let $\mathscr{C}_1$ denote the category of the $m$-PD rings over $R$ whose $m$-PD structure is compatible with $(\b,\gamma)$,
and $\mathscr{C}_2$ the category of the pairs $(A,I)$ consisting of an $R$-algebra $A$ and an ideal $I$ of $A$.
Then, the forgetful functor $\mathscr{C}_1\to\mathscr{C}_2$ has a left adjoint functor.
When $(A,I)$ is an object of $\mathscr{C}_2$, its image under this functor is denoted by $(D_{A,\gamma}^{(m)}(I), \overline{I}, \overline{I}_0, {}^{[\,]})$
and is called the {\rm $m$-PD envelope} of $(A,I)$ $($compatible with $(\b,\gamma)$$)$.
\end{proposition-definition}

\subsection{Crystalline site of level $m$.}
\label{ss:site}

All the arguments in the previous subsection are obviously generalized to
the scheme-theoretical situation.

Let $(S,\fa,\fb,\gamma)$ be an $m$-PD scheme, that is,
a datum which consists of a scheme $S$, a quasi-coherent ideal $\fa$ of $\sO_S$ and
a quasi-coherent $m$-PD structure $(\fb, \gamma)$ on $\a$.
Given two $S$-schemes $X$, $Y$ to which the $m$-PD structure $(\fb,\gamma)$ extends, 
and given an immersion $X\hookrightarrow Y$ over $S$,
we let $D_{X,\gamma}^{(m)}(Y)$ denote the $m$-PD envelope of the immersion.

\vspace{3pt}
Now, we fix throughout this subsection an $S$-scheme $X$,
and assume that the $m$-PD structure $(\fb,\gamma)$ extends to $\sO_X$.
We recall the definition of $m$-crystalline site \cite[4]{LeStum-Quiros:TCHL}.

\begin{definition}
\label{def:site}
    (i) Let $U$ be an open subscheme of $X$.
    An {\em $m$-PD thickening} $(U,T,J,\delta)$ of $U$ over $(S,\a,\b,\gamma)$ is a 
    datum which consists of an $S$-scheme $T$, a closed $S$-immersion $U\hookrightarrow T$
    and an $m$-PD structure $(J,\delta)$ on the ideal of $U\hookrightarrow T$ compatible with $(\b,\gamma)$.

(ii) The {\em $m$-crystalline site} $\Crism(X/S,\fa,\fb,\gamma)$, or $\Crism(X/S)$, is the category of the
$m$-PD thickenings $(U,T,J,\delta)$ of an open subscheme $U$ of $X$ over $(S,\fa,\b,\gamma)$,
morphisms defined in an obvious way,
the topology being defined so that a family
$\{(U_{\lambda}, T_{\lambda}, J_{\lambda}, \delta_{\lambda})\to (U, T, J, \delta)\}_{\lambda}$
of morphisms is a covering if and only if $\{T_{\lambda}\to T\}_{\lambda}$ is a covering with respect to the Zariski topology on $T$.
Its associated topos $(X/S,\fa,\fb,\gamma)\crism$, or $(X/S)\crism$, is called the {\em $m$-crystalline topos}.

(iii) The sheaf of rings
\[
(U,T,J,\delta)\mapsto \Gamma(T,\sO_T)
\]
in the topos $(X/S)\crism$ is called the {\em structure sheaf} of the site $\Crism(X/S)$, and denoted by $\sO_{X/(S,\fa,\fb,\gamma)}^{(m)}$ or by $\sO_{X/S}^{(m)}$. $\square$
\end{definition}

For a sheaf $E\in(X/S)^{(m)}_{\cris}$ and an object $(U,T,J,\delta)$ of $\Cris^{(m)}(X/S)$,
we let $E_{(U,T,J,\delta)}$ denote the sheaf on $T$ induced by $E$.
The most frequent situation is $U=T=X$; in this case, we simply write
$E_X$ instead of $E_{(U,T,J,\delta)}$.

The $m$-crystalline topos has the functoriality which generalizes that of crystalline topos.

\begin{proposition}
    \label{prop:functoriality}
Let $u\colon (S',\fa',\fb',\gamma')\to (S,\fa,\fb,\gamma)$ be an $m$-PD morphism from another $m$-PD scheme,
let $X'$ be an $S'$-scheme such that the $m$-PD structure $(\fb', \gamma')$ extends to $\sO_{X'}$,
and let $g\colon X'\to X$ be a morphism over $S$.
The structure morphism of the $S$-scheme $X$ \resp{$S'$-scheme $X'$} is denoted by $f$ \resp{$f'$}.

Then, there exists a morphism of topoi
\[
g\crism\colon (X'/S')\crism\to (X/S)\crism.
\]
\end{proposition}
\vspace{2pt}

We recall here the construction of this morphism of topoi.
We firstly define a functor
\[
    g^{\ast}\colon \Cris^{(m)}(X/S, \fa, \fb, \gamma) \to (X'/S', \fa', \fb', \gamma')_{\cris}^{(m)}
\]
as follows: for an object $(U, T, J, \delta)$ of $\Cris^{(m)}(X/S)$, the sheaf $g^{\ast}(U, T, J, \delta)$ sends
each object $(U', T', J', \delta')$ of $\Cris^{(m)}(X'/S')$ to the set of $m$-PD morphisms $T'\to T$ over $g$;
here, an ``$m$-PD morphism over $g$'' is defined to be a morphism $h\colon T'\to T$ 
that coincides with $g$ on $U'$, that induces a morphism of closed immersions $(U'\hookrightarrow T')\to(U\hookrightarrow T)$ and
that is an $m$-PD morphism with respect to the $m$-PD structures $(J',\delta')$ and $(J,\delta)$.

The functor $g^{\ast}$ is continuous, the topology of the target being the canonical topology;
the classical proof \cite[III 2.2.2]{Berthelot:CCSCP} works verbatim.
This shows that $g^{\ast}$ extends to the functor
\[
    {g_{\cris}^{(m)}}^{\ast}\colon (X/S, \fa, \fb, \gamma)_{\cris}^{(m)}\to (X'/S', \fa', \fb', \gamma')_{\cris}^{(m)}
\]
that admits a right adjoint ${g_{\cris}^{(m)}}_{\ast}$ \cite[III 1.2]{SGA4}.
We may follow the classical argument \cite[III 2.2]{Berthelot:CCSCP} to show that
the functor ${g_{\cris}^{(m)}}^{\ast}$ is left exact, which gives the morphism of topoi
$g_{\cris}^{(m)}=\big({g_{\cris}^{(m)}}^{\ast}, {g_{\cris}^{(m)}}_{\ast}\big)$.
\vspace{3pt}

Before closing this section, we recall another fundamental morphism
\[
    u_{X/S}^{(m)}\colon (X/S)\crism\to X_{\Zar}
\]
of topoi,
which we call the projection of $m$-crystalline topos to the Zariski topos.
For a sheaf $E$ on $X$, we define the sheaf ${u_{X/S}^{(m)}}^{\ast}(E)$ on $\Cris^{(m)}(X/S)$ to be
the sheaf that sends $(U, T, J, \delta)$ to $E(U)$;
we get a functor ${u_{X/S}^{(m)}}^{\ast}\colon X_{\Zar}\to(X/S)\crism$.
This functor is exact and has a right adjoint ${u_{X/S}^{(m)}}_{\ast}\colon (X/S)\crism\to X_{\Zar}$,
described as in the classical case \cite[III 3.2]{Berthelot:CCSCP}.
Now, we get a morphism $u_{X/S}=\big({u_{X/S}^{(m)}}^{\ast}, {u_{X/S}^{(m)}}_{\ast}\big)$ of topoi.

At last, letting $f\colon X\to S$ denote the structure morphism,
we define the morphism
\[
    f_{X/S}^{(m)}\colon (X/S)\crism\to S_{\Zar}
\]
to be the composite of $u_{X/S}^{(m)}$ and the morphism $f_{\Zar}\colon X_{\Zar}\to S_{\Zar}$.

\subsection{Crystals, differential operators and stratifications.}

In this subsection, we discuss the notion of $m$-crystal, hyper $m$-PD differential operator and hyper $m$-PD stratification.
Let $(S,\fa,\fb,\gamma)$ be an $m$-PD scheme, and $X$ an $S$-scheme such that
the $m$-PD structure $(\fb,\gamma)$ extends to $\sO_X$.
Recall that $p$ is nilpotent on $S$ as we have assumed in Subsection \ref{ss:conventions}.

\begin{definition}
Let $E$ be an $\sO_{X/S}^{(m)}$-module in $(X/S)\crism$.
Then, $E$ is called an {\em $m$-crystal} in $\sO_{X/S}^{(m)}$-modules
if for all morphism $f\colon (U,T,J,\delta)\to (U',T',J',\delta')$
of $\Crism(X/S)$, 
the canonical morphism
\[
f^*(E_{(U',T',J',\delta')})\to E_{(U,T,J,\delta)}
\]
is an isomorphism. $\square$
\end{definition}

In this article, $P_{X/S}^{(m)}$ denotes the $m$-PD envelope of the diagonal immersion $X\hookrightarrow X\times_SX$,
and $\sP_{X/S}^{(m)}$ denotes the structure sheaf of $P_{X/S}^{(m)}$.
We always regard $X\times_SX$ and $P_{X/S}^{(m)}$ as schemes over $X$ by the first projection;
the sheaf $\sO_X$, therefore, acts on $\sO_{X\times_SX}=\sO_X\otimes_{\sO_S}\sO_X$ by multiplication to the first factor,
and this also makes $\sP_{X/S}^{(m)}$ an $\sO_X$-algebra.

\begin{definition}
    Let $M$ and $N$ be two $\sO_X$-modules. Then, a {\em hyper $m$-PD differential operator} from $M$ to $N$
    is an $\sO_X$-linear morphism
    \[
        \sP_{X/S}^{(m)}\otimes_{\sO_X} M\to N.
    \]
    $\square$
\end{definition}

\begin{definition}
Let $M$ be an $\sO_X$-module. Then, a {\em hyper $m$-PD stratification} on $M$ is a $\sP_{X/S}^{(m)}$-linear isomorphism
\[
    \sP_{X/S}^{(m)}\otimes_{\sO_X} M\to M\otimes_{\sO_X} \sP_{X/S}^{(m)}
\]
that induces the identity map on $M$ after passing the quotient $\sP_{X/S}^{(m)}\to\sO_{X}$,
and that satisfies the usual cocycle condition. $\square$
\end{definition}

Now, let $(\fa_0,\fb_0,\gamma_0)$ be a quasi-coherent $m$-PD sub-ideal of $\fa$,
let $S_0\hookrightarrow S$ denote the closed immersion defined by $\fa_0$, and
let $i\colon X_0\hookrightarrow X$ denote its base change by $X\to S$.
We assume that $X$ is smooth over $S$.

The first important proposition in this situation is the following one.

\begin{proposition}
\label{thm:lift}
In the situation above, the functor
\[
{i\crism}_*\colon (X_0/S)\crism\to (X/S)\crism
\]
is exact, where ${i\crism}_{\ast}$ is the direct image functor of the morphism $i\crism$ of topoi induced from
$i$ {\rm (}Proposition \ref{prop:functoriality}{\rm )}. Moreover, the image of the structure sheaf $\sO_{X_0/S}^{(m)}$ under this functor is isomorphic to $\sO_{X/S}^{(m)}$,
and the image of an $m$-crystal in $\sO_{X_0/S}^{(m)}$-modules is an $m$-crystal in $\sO_{X/S}^{(m)}$-modules.
\end{proposition}

\begin{proof}
If $(U,T,J,\delta)$ is an $m$-PD thickening in $\Cris^{(m)}(X/S)$, and if $U_0$ denotes
the fiber product $U\times_X X_0$, then the closed immersion $U_0\hookrightarrow T$ defines an
$m$-PD thickening, which is an object of $\Cris^{(m)}(X_0/S)$.
Let $(U_0,T)$ denote this thickening. Then, $(U_0,T)$ represents $i^{\ast}(U_0, T)$.
Because of the construction and of Yoneda's lemma, if $E$ is an $\sO_{X_0/S}^{(m)}$-module,
$i_{\cris,\ast}(E)(U, T)$ is canonically isomorphic to $E(U_0, T)$,
which gives a canonical isomorphism
\[
    \begin{tikzpicture}
        \matrix(m)[matrix of math nodes,
                     row sep=3em, column sep=2.5em, text height=1.5ex, text depth=0.25ex]
                     { 
                     {i\crism}_\ast(E)_{(U,T,J,\delta)} & E_{(U_0,T)}. \\
                     };
        \path[->] (m-1-1) edge node[above,inner sep=0.5pt] {$\sim$} (m-1-2);
    \end{tikzpicture}
\]
All assertions in the proposition follow from this isomorphism. 
\end{proof}

\begin{remark}
    (This remark is due to P. Berthelot \cite{Berthelot:LAM}.)

    In case $m=0$, Proposition \ref{thm:lift} holds for an arbitrary closed immersion $i\colon X_0\to Y$
    over $S$ to a smooth $S$-scheme $Y$.
    If $m>0$, however, the direct image of an $m$-crystal along such a closed immersion is not necessarily an $m$-crystal.
    This defect can be avoided if we work in the higher-level analogue of the restricted crystalline site \cite[IV.2]{Berthelot:CCSCP}.

    In this article, we do not introduce the ``restricted $m$-crystalline site'' because
    we do not treat arbitrary closed immersions.
\end{remark}    

The following proposition and the corollary is proved as in the classical case
\cite[6.6,6.7]{Berthelot-Ogus:NCC} with the aid of Proposition \ref{thm:lift}.

\begin{proposition}
\label{thm:crystal}
    In the situation above, the following categories are equivalent:

    {\rm (i)} the category of the $m$-crystals in $\sO_{X_0/S}^{(m)}$-modules;

    {\rm (ii)} the category of the $\sO_X$-modules equipped with a hyper $m$-PD stratification.
\end{proposition}

\begin{proof}
    For the sake of completeness of the article, we include here a construction of the equivalence of these categories
    following the classical one \cite[6.6]{Berthelot-Ogus:NCC} applied in our closed immersion.

    First, let $E$ be an $m$-crystal.
    Then, the two projection $p_1,p_2\colon P_{X/S}^{(m)}\to X$ give isomorphisms $p_i^{\ast}E_X\to E_{P_{X/S}^{(m)}}$ ($i=1,2$).
    Then we get an isomorphism $\sP^{(m)}_{X/S}\otimes_{\sO_X} E_X\to E_X\otimes_{\sO_X}\sP^{(m)}_{X/S}$,
    which gives a hyper $m$-PD stratification on the $\sO_X$-module $E_X$.

    Now, let $F$ be an $\sO_X$-module equipped with a hyper $m$-PD stratification.
    We construct an $m$-crystal $E$ in $\sO_{X_0/S}^{(m)}$-modules.
    For each object $(U,T,J,\delta)$ of $\Cris^{(m)}(X/S)$ with a retraction
    $h\colon T\to X$ extending $U\to X$, the sheaf $E_{(U,T,J,\delta)}$ is defined to be $h^{\ast}(\sO_F)$;
    this condition defines a unique $m$-crystal $E$.
\end{proof}

\begin{corollary}
    The functor ${i\crism}_\ast$ in Proposition \ref{thm:lift} induces an equivalence of categories
    from the category of the $m$-crystals in $\sO_{X_0/S}^{(m)}$-modules to
    that of the $m$-crystals in $\sO_{X/S}^{(m)}$-modules.
    A quasi-inverse of this functor is the inverse image functor ${i\crism}^\ast$ of the morphism $i\crism$ of topoi induced from $i$.
    \label{cor:cryst_equiv}
\end{corollary}

\subsection{Linearization.}
\label{ss:linearization}

Here, we discuss the linearization.  
Let $(S,\a,\b,\gamma)$ be an $m$-PD scheme, and $X$ an $S$-scheme such that
the $m$-PD structure $(\fb,\gamma)$ extends to $\sO_X$.

First, $j_X$ is the localization morphism
\[
j_X\colon (X/S)\crism|_X\to (X/S)\crism,
\]
where the source denotes the localized category of $(X/S)\crism$ over the $m$-PD thickening $(X,X,0)$
with the trivial PD structure on $0$.
Then, composing with $u_{X/S}^{(m)}$, we get the morphism of topoi
\[
u_{X/S}^{(m)}|_X\colon (X/S)\crism|_X\to X_{\Zar}.
\]

Now, we define the linearization functor as
\[
L^{(m)}={j_X}_*\circ {u_{X/S}^{(m)}|_X}^*\colon X_{\Zar}\to (X/S)\crism.
\]
The $\sO_X$-module $L^{(m)}(\sF)_X$ is also denoted by $L^{(m)}_X(\sF)$.

\begin{proposition}
\label{thm:linearization}
Assume that $X$ is smooth over $S$, and let $\sF$ be an $\sO_X$-module.

{\rm (i)} $L^{(m)}(\sF)$ is an $m$-crystal, and $L^{(m)}_X(\sF)=\P_{X/S}^{(m)}\otimes_{\sO_X} \sF$.

{\rm (ii)} We have $R{u_{X/S}^{(m)}}_\ast L^{(m)}(\sF)=\sF$.

{\rm (iii)} If $E$ is an $m$-crystal, there exists a canonical isomorphism
\[
    E\otimes_{\sO_{X/S}^{(m)}} L^{(m)}(\sF)\to L^{(m)}(E_X\otimes_{\sO_X}\sF).
\]
\end{proposition}

\begin{proof}
By the construction of $L^{(m)}$, for each $(U,T,J,\delta)\in\Crism(X/S)$ we have
\[
L^{(m)}(\sF)_{(U,T,J,\delta)}={p_T}_*{p_X}^*(\sF),
\]
where $p_T$ \resp{$p_X$} denotes the projection from $D_{U,\gamma}^{(m)}(T\times_S X)$ to $T$ \resp{to $X$}.
This, in particular, shows the latter half of (i).

As for the former half, by following the classical argument \cite[IV 3.1.6]{Berthelot:CCSCP},
it is sufficient to show that the natural morphism
\begin{equation}
\label{eqn:restrictedbasechange}
D^{(m)}_{U,\gamma}(T\times_SX)\to T\times_X P_{X/S}^{(m)}
\end{equation}
is an isomorphism for all $m$-PD thickenings $(U,T,J,\delta)$ such that a retraction $T\to X$ exists.
Now, because $\sP_{X/S}^{(m)}$ is locally isomorphic to an $m$-PD polynomial algebra over $\sO_X$,
the $\sO_T$-algebra $\sO_T\otimes_{\sO_X}\sP_{X/S}^{(m)}$ is an $m$-PD polynomial algebra (over $\sO_T$).
Therefore, the ideal of the closed immersion $U\hookrightarrow T\times_X P_{X/S}^{(m)}$ has an $m$-PD structure
compatible with $\gamma$.
Hence a morphism $T\times_X P_{X/S}^{(m)}\to D^{(m)}_{U,\gamma}(T\times_SX)$ is obtained, and
the universality shows that this is an inverse morphism of (\ref{eqn:restrictedbasechange}).

The proof of (ii) and (iii) can be found in \cite[4]{LeStum-Quiros:EPLCCHL}.
\end{proof}

Now, we set some notations.
For each natural number $r$, let $X^{(r)}_{/S}$ denote the fiber product over $S$ of $r$ copies of $X$.
The $m$-PD envelope of the diagonal immersion $X\hookrightarrow X^{(r+1)}_{/S}$ is denoted by $P^{(m)}_{X/S}(r)$,
and its structure sheaf is denoted by $\sP^{(m)}_{X/S}(r)$.
We have therefore $P^{(m)}_{X/S}=P^{(m)}_{X/S}(1)$ and $\sP_{X/S}^{(m)}=\sP_{X/S}^{(m)}(1)$.
Again, we regard $X^{(r+1)}_{/S}$ and $P_{X/S}^{(m)}(r)$ as schemes over $X$ by the first projection,
which lets $\sO_X$ act on $\sO_{X^{(r+1)}_{/S}}=\sO_X\otimes_{\sO_S}\dots\otimes_{\sO_S}\sO_X$ by multiplication to the first factor
and induce on $\sP_{X/S}^{(m)}(r)$ an $\sO_X$-algebra structure.

For each natural number $r$ and $i\in\{0,\dots,r\}$, let $j_r^i\colon X^{(r+1)}_{/S}\to X$ denote
the $(i+1)$-st projection.
Let $d_r^i\colon P^{(m)}_{X/S}(r+1)\to P^{(m)}_{X/S}(r)$ ($0\leq i\leq r+1$) denote the morphism
corresponding to
\[
(j_{r+1}^0,\dots, j_{r+1}^{i-1}, j_{r+1}^{i+1},\dots, j_{r+1}^{r+1})_S\colon X^{(r+2)}_{/S}\to X^{(r+1)}_{/S},
\]
and $s_r^i\colon P^{(m)}_{X/S}(r-1)\to P^{(m)}_{X/S}(r)$ ($0\leq i\leq r$) the one corresponding to
\[
(j_{r-1}^0,\dots, j_{r-1}^i, j_{r-1}^i,\dots, j_{r-1}^r)_S\colon X^{(r)}_{/S}\to X^{(r+1)}_{/S}.
\]

Then, these data make $P_{X/S}^{(m)}(\bullet)$ a simplicial scheme over $S$, and consequently
$\sP_{X/S}^{(m)}(\bullet)$ is a DGA (differential graded algebra) over $f^{-1}(\sO_S)$;
the differential morphism $d^r\colon\sP_{X/S}^{(m)}(r)\to \sP_{X/S}^{(m)}(r+1)$ is by definition
\[
d^r=\sum_{i=0}^{r+1} (-1)^i{d_r^i}^*.
\]

If $X$ is smooth over $S$, then $\sP_{X/S}^{(m)}$ is isomorphic to the $m$-PD polynomial algebra $\sO_X\{\tau_1,\dots,\tau_n\}_{(m)}$.
Moreover, the graded $\sO_X$-algebra $\sP_{X/S}^{(m)}(\bullet)$ can be identified with the tensor algebra of $\sP_{X/S}^{(m)}$ over $\sO_X$
\cite[Remarque after 2.1.3]{Berthelot:DMAI}; in particular, $\sP_{X/S}^{(m)}(2)$ is isomorphic to $\sP_{X/S}^{(m)}\otimes_{\sO_X}\sP_{X/S}^{(m)}$.

Now, we proceed to a calculation of the hyper $m$-PD stratification on $\sP_{X/S}^{(m)}\otimes_{\sO_X}\sF=L_X^{(m)}(\sF)$,
which exists because of Proposition \ref{thm:linearization} (i) and Proposition \ref{thm:crystal}. 

\begin{lemma}
    \label{lem:inducingstrat}
    Assume that $X$ is smooth over $S$.
    Then, the $m$-PD stratification on $L_X^{(m)}(\sF)=\sP_{X/S}^{(m)}\otimes_{\sO_X}\sF$ is induced by 
    \begin{align*}
    & \sO_{X^{(2)}_{/S}}\otimes (\sO_{X^{(2)}_{/S}} \otimes \sF)\to (\sO_{X^{(2)}_{/S}}\otimes\sF)\otimes\sO_{X^{(2)}_{/S}};\\
    & (1\otimes 1)\otimes(f\otimes g)\otimes x\mapsto (1\otimes g)\otimes x\otimes(1\otimes f),
    \end{align*}
    where the tensor products are taken over $\sO_X$.
\end{lemma}

\begin{proof}
    By the proof of Proposition \ref{thm:crystal} and by Proposition \ref{thm:linearization} (i), the $m$-PD stratification on $L_X^{(m)}(\sF)$ is the composite of the isomorphism
\[
\sP_{X/S}^{(m)}\otimes (\sP_{X/S}^{(m)}\otimes \sF)\to \sP_{X/S}^{(m)}(2)\otimes\sF;\hspace{1em}
(1\otimes 1)\otimes(f\otimes g)\otimes x \mapsto 1\otimes f\otimes g\otimes x
\]
and the inverse of the isomorphism
\[
(\sP_{X/S}^{(m)}\otimes\sF)\otimes\sP_{X/S}^{(m)} \to \sP_{X/S}^{(m)}(2)\otimes\sF;\hspace{1em}
(1\otimes 1)\otimes x\otimes(f\otimes g)\mapsto f\otimes g\otimes 1\otimes x.
\]
Now, an easy observation shows the assertion.
\end{proof}

We omit the proof of the following proposition since it is just a generalization
of the classical argument.

\begin{proposition}
    Let $M$ and $N$ be two $\sO_X$-modules, and $u$ a hyper $m$-PD differential operator from $M$ to $N$.
    Then, the morphism
    \[
    \begin{tikzpicture}
        \matrix (m) [matrix of math nodes,
                     row sep=3em, column sep=3.5em, text height=1.5ex, text depth=0.25ex]
                     { \sP_{X/S}^{(m)}\otimes_{\sO_X} M & \sP_{X/S}^{(m)}\otimes_{\sO_X}\sP_{X/S}^{(m)}\otimes_{\sO_X} M &
                         \sP_{X/S}^{(m)}\otimes_{\sO_X} N \\
                    };
        \path[->,font=\scriptsize]
        (m-1-1) edge node[auto] {${d_1^1}^*\otimes \id$} (m-1-2)
        (m-1-2) edge node[auto] {$\id\otimes u$} (m-1-3);
    \end{tikzpicture}
    \]
    is compatible with the hyper $m$-PD stratifications on both sides viewed as $\P_{X/S}^{(m)}\otimes M=L^{(m)}_X(M)$ and as
    $\P_{X/S}^{(m)}\otimes N=L^{(m)}_X(N)$.
\end{proposition}

\subsection{Differential calculus.}
\label{ss:differential}

We start with a general setting.
Let $(S,\fa,\fb,\gamma)$ be an $m$-PD scheme and $f\colon X\to S$ a morphism
such that the $m$-PD structure $(\fb, \gamma)$ extends to $\sO_X$.

We define the sub-DGA $N\sP_{X/S}^{(m),\bullet}$ by
\begin{equation}
    \label{eqn:NP}
    N\P_{X/S}^{(m),r}=\bigcap_{i=0}^{r+1} \Ker\left({s_r^i}^\ast\right).
\end{equation}

Now, we assume that the morphism $f\colon X\to S$ is smooth, and that
$X$ has global coordinates $t=(t_1,\dots,t_n)$ over $S$;
we set $\tau_i:=d^0(t_i)\in N\sP_{X/S}^{(m),1}$, which is the pullback of $t_i\otimes 1-1\otimes t_i$
by the natural morphism $P_{X/S}^{(m)}\to X\times_SX$; for $I=(i_1,\dots,i_n)\in\bN^n$, we set
\[
\tau^{\{I\}}:=\prod_{j=1}^n\tau_j^{\{i_j\}}.
\]
Under this notation, $N\sP^{(m),1}_{X/S}$ is freely generated by $\{\tau^{\{I\}}\}_{I\in\bN^n\setminus\{0\}}$,
and the DGA $N\sP_{X/S}^{(m),\bullet}$ is isomorphic to the tensor algebra
of $N\sP^{(m),1}_{X/S}$ over $\sO_X$.

We keep these assumptions and notations in the remainder of this subsection.

\begin{proposition}
\label{thm:Pm}
For all $I\in\bN^n$, the morphism $d^1\colon \sP_{X/S}^{(m)}(1)\to \sP_{X/S}^{(m)}(2)$ satisfies
\[
d^1\left(\tau^{\{I\}}\right)=-\sum_{0<V<I}\qbinom{I}{V}\tau^{\{V\}}\otimes\tau^{\{I-V\}}.
\]
\end{proposition}

\begin{proof}
Let us recall that $d^1$ is by definition equal to ${d_1^0}^*-{d_1^1}^*+{d_1^2}^*$.
We have
\[
{d_1^0}^*\left(\tau^{\{I\}}\right)=\tau^{\{I\}}\otimes 1 \hspace{5pt}\mathrm{and}\hspace{5pt} {d_1^2}^*\left(\tau^{\{I\}}\right)=1\otimes\tau^{\{I\}}.
\]
By using the equation ${d_1^1}^*(\tau_i)=\tau_i\otimes 1+1\otimes\tau_i$ and Proposition \ref{prop:pdbinom}, we calculate
\begin{eqnarray*}
    {d_1^1}^*\left(\tau^{\{I\}}\right) = (\tau\otimes 1+1\otimes\tau)^{\{I\}}
    &=& \sum_{0\leq V\leq I}\qbinom{I}{V}(\tau\otimes 1)^{\{V\}}(1\otimes\tau)^{\{I-V\}}\\
    &=& \sum_{0\leq V\leq I}\qbinom{I}{V}\tau^{\{V\}}\otimes\tau^{\{I-V\}}.
\end{eqnarray*}
This completes the proof. \end{proof}
\vspace{2pt}

The hyper $m$-PD stratification on $\sP_{X/S}^{(m)}=L^{(m)}_X(\sO_X)$ is described
in the following proposition.

\begin{proposition}
\label{thm:hpd_of_Pm}
The hyper $m$-PD stratification
\[
    \P_{X/S}^{(m)}\otimes_{\sO_X}\P_{X/S}^{(m)}\to \P_{X/S}^{(m)}\otimes_{\sO_X}\P_{X/S}^{(m)}
\]
of $\P_{X/S}^{(m)}$ maps $1\otimes\tau^{\{I\}}$ to
\[
\sum_{0\leq V\leq I}\qbinom{I}{V}\,\tau^{\{V\}}\otimes(-\tau)^{\{I-V\}}.
\]
\end{proposition}

\begin{proof}
By Lemma \ref{lem:inducingstrat}, the hyper $m$-PD stratification is induced by 
\[
    \sO_{X^{(2)}_{/S}}\otimes_{\sO_X}\sO_{X^{(2)}_{/S}} \to \sO_{X^{(2)}_{/S}}\otimes_{\sO_X}\sO_{X^{(2)}_{/S}};\quad 
    (1\otimes 1)\otimes(f\otimes g)\mapsto (1\otimes g)\otimes(1\otimes f).
\]
Noting that $(t_i\otimes 1)\otimes(1\otimes 1)=(1\otimes 1)\otimes(t_i\otimes 1)$, we see that this morphism sends $(1\otimes 1)\otimes(t_i\otimes 1-1\otimes t_i)$ to
\[
(t_i\otimes 1-1\otimes t_i)\otimes(1\otimes 1)-(1\otimes 1)\otimes(t_i\otimes 1-1\otimes t_i),
\]
therefore the hyper $m$-PD stratification sends $1\otimes\tau^{\{I\}}$ to
\[
(\tau\otimes 1-1\otimes \tau)^{\{I\}}.
\]
Now, we may calculate as in the proof of Proposition \ref{thm:Pm} to complete the proof.
\end{proof}

Similarly, starting from the simplicial scheme $P_{X/S}^{(m)}(\bullet +1)$ over $X$,
we construct a DGA $L\sP_{X/S}^{(m)}(\bullet):=P_{X/S}^{(m)}(\bullet +1)$ over $\sO_X$ and
its sub-DGA
\begin{equation}
    \label{eqn:LNP}
    LN\sP_{X/S}^{(m),\bullet}:=\bigcap_{i=1}^{r+2}\Ker({s_{r+1}^i}^*)
\end{equation}
Note that the differential morphism $d^r\colon L\sP_{X/S}^{(m)}(r)\to L\sP_{X/S}^{(m)}(r+1)$ is by definition
\begin{equation}
    \label{eqn:diff_of_LP}
d^r=\sum_{i=1}^{r+2}(-1)^{i+1}{d_{r+1}^i}^{\ast}.
\end{equation}

By Proposition \ref{thm:linearization}, each $L\sP_{X/S}^{(m)}(r)=L_X^{(m)}(\sP_{X/S}^{(m)}(r))$ has a hyper $m$-PD stratification.

\begin{lemma}
    \label{lem:hpd_of_P}
    For each natural number $r$, the differential morphism $d^r\colon L\sP_{X/S}^{(m)}(r)\to L\sP_{X/S}^{(m)}(r+1)$
    is compatible with the hyper $m$-PD stratification on both sides.
\end{lemma}

\begin{proof}
    By Lemma \ref{lem:inducingstrat}, the hyper $m$-PD stratification on $L\sP_{X/S}^{(m)}(r)$ is induced by
    \[ 
    \begin{tikzpicture}
        \matrix (m) [matrix of math nodes,
                     row sep=0.5em, column sep=2em, text height=1.5ex, text depth=0.25ex]
                     { 
                         \sO_{X^{(2)}_{/S}}\otimes_{\sO_X}\sO_{X^{(r+2)}_{/S}} & \sO_{X^{(r+2)}_{/S}}\otimes_{\sO_X}\sO_{X^{(2)}_{/S}}\\
                     (1\otimes 1)\otimes (f\otimes g\otimes h_1\otimes\dots\otimes h_r) & (1\otimes g\otimes h_1\otimes\dots\otimes h_r)\otimes (1\otimes f). \\
                     };
        \path[->] (m-1-1) edge (m-1-2);
        \path[|->] (m-2-1) edge (m-2-2);
    \end{tikzpicture}
    \]
    The morphism ${d_{r+1}^i}^*$ for $i=1,\dots,r+1$ is therefore compatible with this morphism,
    and so is $d^r$ because of its definition (\ref{eqn:diff_of_LP}).
\end{proof}

%
%
%
%
%
%

\section{Higher Poincar\'e Lemma --- Local Results.}

\subsection{Higher de Rham complex.}
\label{ss:higherdR}

Throughout Section 2, we fix an $m$-PD scheme $(S,\fa,\fb,\gamma)$ and a smooth $S$-scheme $X$
that has global coordinates $t=(t_1,\dots,t_n)$; precisely,
we assume that there exists an \'etale morphism $g\colon X\to \bA_S^n=\Spec \sO_S[t'_1,\dots,t'_n]$ over $S$ and 
we write the $t_i=g^{\ast}(t'_i)$ for $i=1,\dots,n$.

First, in this subsection, we introduce the ``higher de Rham complex'' in this situation.
Recall from (\ref{eqn:NP}) and (\ref{eqn:LNP}) the definition of $N\sP_{X/S}^{(m),\bullet}$
and $LN\sP_{X/S}^{(m),\bullet}$

Let $\sK_{X/S}^{(m),\bullet}$ be the DG-ideal of $N\sP_{X/S}^{(m),\bullet}$ generated by $\tau^{\{I\}}\in N\sP_{X/S}^{(m),1}$'s, where $I$ runs through all multi-indices
in $\bN^n\setminus\{0,p^m\one_1,\dots,p^m\one_n\}$.
It is easy to show that the ideal $\sP_{X/S}^{(m)}\otimes\sK_{X/S}^{(m),\bullet}$ of $LN\sP_{X/S}^{(m),\bullet}$
is a DG-ideal.

\begin{definition}
    The {\em higher de Rham complex} $\hO_{X/S}^{(m),\bullet}$ is by definition
    the quotient of $N\P^{(m),\bullet}_{X/S}$ by the DG-ideal $\sK_{X/S}^{(m),\bullet}$ above.

The {\em linearized higher de Rham complex} $L\hO^{(m),\bullet}_{X/S}$ is by definition the quotient
of $LN\P^{(m)}_{X/S}$ by $\sP_{X/S}^{(m)}\otimes\sK_{X/S}^{(m),\bullet}$.
$\square$
\label{def:higherdR}
\end{definition}

It should be remarked that these complexes essentially depend on the choice of the system of global coordinates on $X$.
This construction therefore can not be generalized to the global situation.

For $i=1,\dots,n$, the image of $\tau_i^{p^m}$ under the natural surjection $N\sP_{X/S}^{(m),1}\to\hO_{X/S}^{(m),1}$ is denoted by $\bar{\tau}_i^{p^m}$.

\begin{proposition}

    {\rm (i)} $\hO_{X/S}^{(m),1}$ is a free $\sO_X$-module of rank $n$ with basis $\big\{\bar{\tau}_j^{p^m}\big\}_{j=1,\dots,n}$.

    {\rm (ii)} $\hO_{X/S}^{(m),\bullet}$ is isomorphic to the exterior algebra of $\hO_{X/S}^{(m),1}$ as a graded $\sO_X$-module.

    {\rm (iii)} The differential map $d^r$ sends the section $\tau^{\{I\}}\otimes\bar{\tau}_{j_1}^{p^m}\wedge\dots\wedge\bar{\tau}_{j_r}^{p^m}$
    of $L\hO^{(m),r}_{X/S}$, where $I=(i_1,\dots,i_n)\in\bN^n$, to the section
    \begin{equation}
    \sum_{\substack{j=1,\dots,n\\ i_j\geq p^m}}\qbinom{i_j}{p^m}\tau^{\{I-p^m\one_j\}}\otimes\bar{\tau}_j^{p^m}\wedge\bar{\tau}_{j_1}^{p^m}\wedge\dots\wedge\overline{\tau}_{j_r}^{p^m}.
    \label{eqn:diff_sent}
    \end{equation}
\label{prop:structure}
\end{proposition}

\begin{proof}
The $\sO_X$-module $\hO^{(m),r}_{X/S}$ is generated by the sections
\[
\bar{\tau}_{j_1}^{p^m}\otimes\bar{\tau}_{j_2}^{p^m}\otimes\dots\otimes\bar{\tau}_{j_r}^{p^m},
\]
and their relations are generated by
\[
\sum_{\substack{0<V<I\\ V=p^m\one_i, I-V=p^m\one_j}}\qbinom{I}{V}\,\bar{\tau}_i^{p^m}\otimes\bar{\tau}_j^{p^m}=0.
\]
For $I=2p^m\one_i$, this gives
\[
\bar{\tau}_i^{p^m}\otimes\bar{\tau}_i^{p^m}=0
\]
because the coefficient $\qbinom{2p^m}{p^m}$ is invertible by Lemma \ref{lem:invert} below.
Next, for $I=p^m\one_i+p^m\one_j$ with $i\neq j$, it gives
\[
\bar{\tau}_i^{p^m}\otimes\bar{\tau}_j^{p^m}+\bar{\tau}_j^{p^m}\otimes\bar{\tau}_i^{p^m}=0.
\]
For the other $I$'s, the given relations are trivial,
which shows (i) and (ii).
Now, (iii) is a direct consequence of Proposition \ref{thm:Pm}.
\end{proof}

\begin{lemma}
	If $i\geq p^m$, the number $\qbinom{i}{p^m}$ is congruent to $1$ modulo $p$.
	\label{lem:invert}
\end{lemma}

\begin{proof}
Put $i=p^mq+r$ with $q$ a natural number and $0\leq r<p^m$.
Then, we have $\qbinom{i}{p^m}=(1/q)\binom{i}{p^m}$, and this equals
\[
\frac{1}{q}\frac{p^mq+r}{r}\frac{p^mq+r-1}{r-1}\dots\frac{p^mq+1}{1}\frac{p^mq}{p^m}\frac{p^mq-1}{p^m-1}\dots\frac{p^mq-p^m+r+1}{r+1}.
\]
For $0<k<p^m$, the number $(p^mq+k)/k$ is congruent to $1$ modulo $p$.
This shows the lemma.
\end{proof}

\subsection{Formal Higher Poincar\'e lemma.}
\label{ss:formalPL}

Now, we establish the Poincar\'e lemma for the higher de Rham complex
in the local situation specified in the top of this section.

\begin{lemma}
\label{lem:formal}
The linearized higher de Rham complex $L\hO^{(m),\bullet}_{X/S}$ is a resolution of 
the direct sum of $p^{mn}$ copies of $\sO_X$. More strongly, if $h\colon T\to X$ is a morphism of $S$-schemes,
the $\sO_T$-linear map
\[
\iota'\colon \bigoplus_{I\in\fB_n^{(m)}}\sO_Te_I\to h^*\left(L\hO^{(m),\bullet}_{X/S}\right);~e_I\mapsto h^*(\tau^{\{I\}})
\]
is a quasi-isomorphism, where $\fB_n^{(m)}$ denotes the set of elements $I=(i_1,\dots,i_n)\in\bN^n$ such that
$i_j\leq p^m$ for all $j=1,\dots,n$.
\end{lemma}

\begin{proof}
When $n=1$ we have to show that the sequence
\[
0\to \bigoplus^{p^m-1}_{i=0}\sO_Te_i\to \sO_T\{\tau\}_{(m)}\to \sO_T\{\tau\}_{(m)}\bar{\tau}^{p^m}\to 0
\]
is exact, where $\tau$ here denotes $h^\ast(\tau_1)$;
the second morphism sends $\tau^{\{i\}}$ to zero if $i<p^m$ and to $\qbinom{i}{p^m}\bar{\tau}^{p^m}$ if $i\geq p^m$.
Since $\qbinom{i}{p^m}$ is invertible by Lemma \ref{lem:invert}, the exactitude follows.

For an arbitrary $n$, the morphism $\iota'$ is the tensor product of that for the $(n-1)$-dimensional case
and that for the $1$-dimensional case. Since each term of these complexes is free, the proof is obtained by induction on $n$.
\end{proof}

The following proposition is a direct consequence of Lemma \ref{lem:formal}.

\begin{proposition}
	\label{thm:formal}
    Using the isomorphism
    \[
    \beta\colon \bigoplus_{I\in\fB_n^{(m)}}\sO_Xe_I\to \bigoplus_{I\in\fB_n^{(m)}}\sO_Xe_I;
    \hspace{8pt} e_I\mapsto \sum_{0\leq J\leq I}\binom IJ t^Je_{I-J},
    \]
    we define the morphism
	\[
    \iota:=\iota'\circ\beta^{-1} \colon \bigoplus_{I\in\fB_n^{(m)}}\sO_Xe_I\to L\hO^{(m),\bullet}_{X/S}.
	\]
    Then, this is a quasi-isomorphism.
\end{proposition}

\subsection{Higher Poincar\'e lemma.}
\label{ss:higherPL}

In this subsection, we ``lift'' the results in the previous subsection to the $m$-crystalline site.

Set
\[
\sF:=\bigoplus_{I\in\fB_n^{(m)}}\sO_Xe_I,
\]
and equip it the hyper $m$-PD stratification
\[
    \sP_{X/S}^{(m)}\otimes_{\sO_X}\sF\to \sF\otimes_{\sO_X}\sP_{X/S}^{(m)};\hspace{3pt}1\otimes e_I\mapsto e_I\otimes 1.
\]
Let $F$ denote the $m$-crystal corresponding, by Proposition \ref{thm:crystal}, to the module $\sF$ and this hyper $m$-PD stratification.
$F$ is obviously isomorphic to the direct sum of $p^{mn}$ copies of the structure sheaf $\sO_{X/S}^{(m)}$.

\begin{lemma}
    The $\sO_X$-linear morphism $\iota\colon\sF\to L\hO_{X/S}^{(m),0}=\sP_{X/S}^{(m)}$ defined in Proposition \ref{thm:formal} is compatible with the
	hyper $m$-PD stratifications on both sides.
	\label{lem:horizontal_1}
\end{lemma}

\begin{proof}
This lemma states that the diagram
\[
\begin{tikzpicture}[description/.style={fill=white,inner sep=2pt}]
    \matrix (m) [matrix of math nodes,
                 row sep=3em, column sep=3.5em, text height=1.5ex, text depth=0.4ex]
        { \sP_{X/S}^{(m)}\otimes \sF & \sP_{X/S}^{(m)}\otimes\sP_{X/S}^{(m)} \\
          \sF\otimes\sP_{X/S}^{(m)} & \sP_{X/S}^{(m)}\otimes\sP_{X/S}^{(m)} \\
        };
    \path[->,font=\scriptsize]
        (m-1-1) edge node[auto] {$1\otimes\iota$} (m-1-2)
                edge (m-2-1)
        (m-1-2) edge (m-2-2)
        (m-2-1) edge node[auto] {$\iota\otimes 1$} (m-2-2);
\end{tikzpicture}
\]
is commutative, where the right vertical morphism is calculated by using Proposition \ref{thm:hpd_of_Pm}.
In order to prove the commutativity, let us consider the section
\begin{equation}
	\label{eqn:new_basis}
	1\otimes\left(\sum_{0\leq J\leq I}\binom IJt^Je_{I-J}\right)
\end{equation}
of $\P_{X/S}^{(m)}\otimes\sF$ for each $I\in\fB_n^{(m)}$; when $I$ runs through all elements in $\fB_n^{(m)}$,
the sections (\ref{eqn:new_basis}) form a $\P_{X/S}^{(m)}$-basis of this module.
This section (\ref{eqn:new_basis}) is sent by $1\otimes\iota$ to $1\otimes\tau^{I}$, and its image in the bottom is
\[
\sum_{0\leq J\leq I}\binom IJ\tau^{I-J}\otimes(-\tau)^J,
\]
which is the image by $\iota\otimes 1$ of
\begin{equation}
	\label{eqn:section_1}
    \sum_{0\leq J\leq I}\binom{I}{\,J\,}\left\{\sum_{0\leq K\leq I-J}\binom{I-J}{K}t^Ke_{I-J-K}\right\}\otimes (-\tau)^J
\end{equation}
(note that $\qbinom IV=\binom IV$ and that $\tau^{\{V\}}=\tau^V$ etc.\ because $I\in\fB_n^{(m)}$).
On the other hand, (\ref{eqn:new_basis}) goes down to the section
\begin{equation}
	\label{eqn:section_2}
\sum_{0\leq J\leq I}\binom IJ e_{I-J}\otimes t^J.
\end{equation}
Therefore, the problem is showing that the two sections (\ref{eqn:section_1}) and (\ref{eqn:section_2}) are identical.
Now, the section (\ref{eqn:section_1}) equals
\[
\sum_{0\leq L\leq I}\binom{I}{\,L\,}e_{I-L}\left\{\sum_{0\leq J\leq L}\binom{L}{\,J\,}(t\otimes 1)^{L-J}(-\tau)^J\right\}
\]
by changing the variables as $L=J+K$.
The section of $\sP^{(m)}_{X/S}$ in the braces is the canonical image of the section
\[
\sum_{0\leq J\leq L}\binom{L}{\,J\,}(t\otimes 1)^{L-J}(1\otimes t-t\otimes 1)^{J}=(1\otimes t)^L
\]
of $\sO_{X^{(2)}_{/S}}$. This completes the proof.
\end{proof}

Now, we are ready to prove the higher Poincar\'e lemma.

\begin{theorem}
    \label{thm:resolution}
	Let $M$ be an $\sO_{X/S}^{(m)}$-module.
    Then, $M\otimes_{\sO^{(m)}_{X/S}} L^{(m)}(\hO^{(m),\bullet}_{X/S})$ forms a complex of $\sO_{X/S}^{(m)}$-modules
    that resolves the direct sum of $p^{mn}$ copies of $M$.
\end{theorem}

\begin{proof}
    First, we have a complex $L^{(m)}(\hO_{X/S}^{(m),\bullet})$ of $\sO_{X/S}^{(m)}$-modules
that gives $L\hO_{X/S}^{(m),\bullet}$ on $X$;
each term is, in fact, calculated by Proposition \ref{thm:linearization} (i), and
    each differential of $L\hO_{X/S}^{(m),\bullet}$ is compatible with
    the hyper $m$-PD stratification on each term by Lemma \ref{lem:hpd_of_P}.
    Moreover, Lemma \ref{lem:horizontal_1} ensures the $\sO_{X/S}^{(m)}$-linear map $F\to L^{(m)}(\hO_{X/S}^{(m),0})$,
    whose composition with the differential map $L^{(m)}(\hO_{X/S}^{(m),0})\to L^{(m)}(\hO_{X/S}^{(m),1})$ is zero by Proposition \ref{thm:formal}.
    By tensoring $M$ over $\sO_{X/S}^{(m)}$, we get a morphism
\begin{equation}
    M\otimes F \to M\otimes L^{(m)}(\hO_{X/S}^{(m),\bullet}).
\label{eq:q-isom}
\end{equation}
We show that this is a quasi-isomorphism.  

It suffices to argue on each $m$-PD thickening $(U,T,J,\delta)$ in $\Crism(X/S)$,
and then the assertion is local on $T$.
We may therefore assume that there exists an $S$-morphism $h\colon T\to X$ compatible with $i\colon U\hookrightarrow T$.
Then on $T$, the map (\ref{eq:q-isom}) is written as
\[
    \bigoplus_{I\in\fB_n^{(m)}} M_T\,e_I\to M_T\otimes_{\sO_T} h^*\left(L\hO_{X/S}^{(m),\bullet}\right).
\]
Now, Proposition \ref{thm:formal} shows that
\[
    \bigoplus_{I\in\fB_n^{(m)}} \sO_T\,e_I = \sO_T\otimes_{\sO_X} \sF\to h^*\left(L\hO_{X/S}^{(m),\bullet}\right)
\]
is quasi-isomorphic, and even after $M_T$ is tensored, 
it remains quasi-isomorphic because each term is locally free.
This shows the assertion.
\end{proof}

\begin{corollary}\label{thm:poincare}
    Let $(S,\fa,\fb,\gamma)$ be an $m$-PD scheme and $X$ a smooth $S$-scheme that has global coordinates
    \footnote{This assumption, which we have always kept in this section,
    is written here just for the convienience of the reader.}.
	Let $E$ be an $m$-crystal in $\sO_{X/S}^{(m)}$-modules.
Then, there exists an isomorphism in the derived category
\[
    \left(\bR{u_{X/S}^{(m)}}_*(E)\right)^{\oplus p^{mn}}\to E_X\otimes_{\sO_X} \hO^{(m),\bullet}_{X/S}.
\]
\end{corollary}

\begin{proof}
The previous theorem, after the functor $\bR{u_{X/S}^{(m)}}_*$ is applied, shows that it suffices to prove that
\[
\bR{u_{X/S}^{(m)}}_*\left(E\otimes L^{(m)}(\hO_{X/S}^{(m),\bullet})\right)=E_X\otimes\hO_{X/S}^{(m),\bullet},
\]
and we know from Proposition \ref{thm:linearization} (ii) and (iii) that this is true as graded $\sP^{(m)}_{X/S}$-modules.
Since this identification is via the map
\[
{d^0}^*(E_X\otimes\hO_{X/S}^{(m),\bullet})\to \P_{X/S}^{(m)}\otimes E_X\otimes\hO_{X/S}^{(m),\bullet},
\]
we see that the differential maps on both sides coincide because of the relations of $d_i^*$'s. 
\end{proof}

\section{Finiteness of Cohomology.}

\subsection{Cohomological Boundedness.}
\label{ss:application}

Now, we are going to use the higher Poincar\'e lemma to prove the boundedness,
the base change and the finiteness of the crystalline cohomology of level $m$.

\begin{theorem}
\label{thm:bounded}
Let $(S,\fa,\fb,\gamma)$ be an $m$-PD scheme, $(\fa_0, \fb_0, \gamma_0)$ be
a quasi-coherent $m$-PD sub-ideal of $\fa$.
Let $X$ be a smooth scheme over $S_0$, and assume that the structure morphism $f\colon X\to S$ is quasi-compact and quasi-separated;
moreover, we assume that $S$ is quasi-separated.
Then, for each quasi-coherent $\sO_{X/S}^{(m)}$-module $E$ and each natural number $i$,
the $\sO_S$-module $R^i{f_{X/S}^{(m)}}_*E$ is quasi-coherent.
Moreover, there exists an integer $r$ such that $R^i{f_{X/S}^{(m)}}_*E=0$ for all $i>r$ and for all quasi-coherent $\sO_{X/S}^{(m)}$-module $E$.
\end{theorem}

\begin{proof}
    First, if $X$ can be lifted to a smooth $S$-scheme having local coordinates, then Corollary \ref{thm:poincare} and Proposition \ref{thm:lift} show the assertions.

Next, we assume that $X$ is separated. Let $U=(U_i)$ be a finite covering
by affine open subschemes of $X$ which have local coordinates.
For each natural number $\nu$, put
\[
U_{(\nu)}:=\coprod_{i_0<\dots<i_{\nu}}U_{i_0}\cap\dots\cap U_{i_{\nu}},
\]
and let $j_{(\nu)}$ denote the natural map $U_{(\nu)}\to X$.
Then \cite[7.6]{Berthelot-Ogus:NCC}, there exists a spectral sequence
\[
E_1^{p,q}=R^q{f^{(m)}_{U_{(p)}/S}}_\ast j_{(p)}^\ast(E)\Longrightarrow R^{n}{f^{(m)}_{X/S}}_\ast E,
\]
by which the assertions are reduced to those for each $U_{i_0}\cap\dots\cap U_{i_p}$.
These schemes in turn are affine and have local coordinates by assumption,
therefore the proof is finished in this case.

At last, for general $X$, we repeat the same argument;
the schemes $U_{i_0}\cap\dots\cap U_{i_p}$ constructed as above are not necessarily affine, but are quasi-compact and quasi-affine, therefore separated.
\end{proof}

\subsection{Base Change Theorem.}

\begin{theorem}
\label{thm:basechange}
Let $u\colon(S',\a',\b',\gamma')\to(S,\a,\b,\gamma)$ be a morphism of $m$-PD schemes,
$Y$ \resp{$Y'$} a scheme over $S$ \resp{$S'$} and $h\colon Y'\to Y$ an $S$-morphism.
We assume that $Y$ is quasi-compact.
Let $f\colon X\to Y$ be a smooth morphism, and let $f'\colon X'\to Y'$ denote the base change of $f$ over $h$,
and let $g$ denote the projection $X'\to X$.
Then, if $E$ is a flat and quasi-coherent $\sO_{X/S}^{(m)}$-module, there exists an isomorphism
\[
\bL h\crism^*R f\crism_*(E)\to R f'\crism_* g\crism^*(E).
\]
\end{theorem}

\[
\begin{tikzpicture}
    \matrix (m) [matrix of math nodes,
                 row sep=1em, column sep=1.5em, text height=1.5ex, text depth=0.25ex]
         { X' & & X \\
              & \square & \\
           Y' & & Y \\[0.5em]
           (S',\fa',\fb',\gamma') & & (S,\fa,\fb,\gamma) \\
         };
    \path[->,font=\scriptsize]
        (m-1-1) edge node[auto] { $g$ } (m-1-3)
                edge node[auto,swap] { $f'$ } (m-3-1)
        (m-1-3) edge node[auto] { $f$ } (m-3-3)
        (m-3-1) edge node[auto] { $h$ } (m-3-3)
                edge (m-4-1)
        (m-3-3) edge (m-4-3)
        (m-4-1) edge (m-4-3);
\end{tikzpicture}
\]

\begin{proof}
Because $\bR{f\crism}_*(E)$ is bounded by Theorem \ref{thm:bounded}, the complex in the left hand side makes sense.
Then, we may construct the base change morphism using the adjunction formula \cite[V 3.3.1]{Berthelot:CCSCP}.
In order to prove that this morphism is isomorphic, following the argument for classical case \cite[V 3.5.5]{Berthelot:CCSCP}, we know that it suffices to prove the following weaker proposition. 
\end{proof}

\begin{proposition}
\label{thm:wbc}
In the situation in Theorem \ref{thm:basechange}, let $(\fa_0,\fb_0,\gamma_0)$ \resp{$(\fa_0',\fb_0',\gamma_0')$} be a quasi-coherent $m$-PD sub-ideal of $\fa$ \resp{$\fa'$},
and assume that $Y$ \resp{$Y'$} is the closed subscheme $S_0$ \resp{$S'_0$}
of $S$ \resp{$S'$} defined by the ideal $\fa_0$ \resp{$\fa_0'$}.
We rename the morphisms as in the diagram below.
Then, if $E$ is a flat and quasi-coherent $\sO_{X/S}^{(m)}$-module,
the base changing map
\begin{equation}
    \bL u^\ast R{f^{(m)}_{X/S}}_\ast(E)\to R{f'^{(m)}_{X'/S'}}_\ast g\crism^\ast(E)
\label{eq:bcmap}
\end{equation}
is an isomorphism.
\begin{equation*}
    \begin{tikzpicture}
        \matrix (m) [matrix of math nodes, row sep=1em, column sep=1.5em, text height=1.5ex, text depth=0.25ex]
        { 
        X' &[1em] &[1em] X &[1em] \\
          & \square & & \\
        S'_0 & & S_0 & \\[-0.4em]
        & S' & & S \\
        };
    \path[->,font=\scriptsize]
        (m-1-1) edge node[auto] { $g$ } (m-1-3)
                edge node[auto,swap] { $f_0'$ } (m-3-1)
        (m-1-3) edge node[auto,swap] { $f_0$ } (m-3-3)
                edge node[auto] { $f$ } (m-4-4)
        (m-4-2) edge node[auto] { $u$ } (m-4-4)
        (m-3-1) edge node[auto] { $u_0$ } (m-3-3)
        (m-1-1) edge[preaction={draw=white,-,line width=6pt}] node[auto] { $f'$ } (m-4-2);
    \path[right hook->]
        (m-3-1.south east) edge (m-4-2)
        (m-3-3) edge (m-4-4);
    \end{tikzpicture}
\end{equation*}
\end{proposition}

\begin{proof}
At first, we assume that $X$ is lifted to a scheme $Y$ which is affine, is smooth of relative dimension $n$ over $S$, and has local coordinates.
Then by Corollary \ref{cor:cryst_equiv}, we may assume that $X=Y$.

Now, by Corollary \ref{thm:poincare}, we have an isomorphism
\begin{equation}
    \begin{tikzpicture}
        \matrix (m) [matrix of math nodes,
                     row sep=3em, column sep=2.5em, text height=1.5ex, text depth=0.25ex]
                     { 
                     \left(\bR{f^{(m)}_{X/S}}_\ast(E)\right)^{\oplus p^{mn}} & 
                     \bR f_\ast(E_X\otimes\hO_{X/S}^{(m),\bullet}), \\
                     };
        \path[->] (m-1-1) edge node[above,inner sep=0.5pt] {$\sim$} (m-1-2);
    \end{tikzpicture}
	\label{eqn:1st}
\end{equation}
and $E':=g\crism^*(E)$ in turn satisfies 
\begin{equation}
    \begin{tikzpicture}
        \matrix (m) [matrix of math nodes,
                     row sep=3em, column sep=2.5em, text height=1.5ex, text depth=0.25ex]
                     { 
                     \left(\bR{f'^{(m)}_{X'/S'}}_\ast(E')\right)^{\oplus p^{mn}} & 
                     \bR f'_\ast(E'_{X'}\otimes\hO_{X'/S'}^{(m),\bullet}). \\
                     };
        \path[->] (m-1-1) edge node[above,inner sep=0.5pt] {$\sim$} (m-1-2);
    \end{tikzpicture}
	\label{eqn:2nd}
\end{equation}

In the right-hand side of (\ref{eqn:1st}) \resp{(\ref{eqn:2nd})}, the functor $\bR f_*$ \resp{$\bR f'_*$} can be replaced by $f_*$ \resp{$f'_*$} because of the quasi-coherence of $E$ \resp{$E'$}.
We then have a commutative diagram
\begin{equation*}
    \begin{tikzpicture}
        \matrix (m) [matrix of math nodes, row sep=3em, column sep=3.5em, text height=1.5ex, text depth=0.8ex]
        { 
        \bL u^*\left(f_*(E_X\otimes\hO^{(m),\bullet}_{X/S})\right) &
        f'_*\left(E'_{X'}\otimes\hO^{(m),\bullet}_{X'/S'}\right) \\
        \left(\bL u^*\bR{f^{(m)}_{X/S}}_*(E)\right)^{\oplus p^{mn}} &
        \left(\bR{f'^{(m)}_{X'/S'}}_*(E')\right)^{\oplus p^{mn}}, \\
        };
        \path[->] (m-1-1) edge node[above,inner sep=0.5pt] {$\sim$} (m-1-2);
        \path[->,font=\scriptsize]
            (m-2-1) edge node[auto] { $\bL u^*(\ref{eqn:1st})$ } (m-1-1)
            (m-2-1) edge (m-2-2)
            (m-2-2) edge node[auto,swap] { (\ref{eqn:2nd}) } (m-1-2);
    \end{tikzpicture}
\end{equation*}
where the lower horizontal morphism is the direct sum of $p^{mn}$ copies of (\ref{eq:bcmap}).
The upper horizontal morphism exists because the local coordinates of $X$ over $S$ and those of $X'$ over $S'$ are compatible,
and it is clear that this is an isomorphism.

This completes the proof for this special case.
In the general case, we can use the descent argument \cite[pp.344-347]{Berthelot:CCSCP}.
\end{proof}

\subsection{Finiteness.}
\begin{theorem}
    \label{thm:finiteness}
Let $(S,\a,\b,\gamma)$ be an $m$-PD scheme and $(\fa_0,\fb_0,\gamma_0)$ a quasi-coherent $m$-PD sub-ideal of $\a$.
The closed subscheme of $S$ defined by $\fa_0$ is denoted by $S_0$.
Assume that $X$ is a smooth proper scheme over $S_0$ and that $S$ is noetherian.
Let $f$ denote the structure morphism of $X$ over $S$.
In this situation, if $E$ is a locally free $\sO^{(m)}_{X/S}$-module of finite rank,
then $\bR{f^{(m)}_{X/S}}_*E$ is a perfect complex of $\sO_S$-modules.
\end{theorem}

\begin{proof}
The assertion is equivalent to showing that the complex of $\sO_S$-modules $\bR{f^{(m)}_{X/S}}_\ast E$ is
pseudo-coherent and locally of finite Tor-dimension \cite[I 5.8.1]{SGA6}
(recall that a complex $K$ of $\sO_S$-modules is pseudo-coherent if and only if it is locally of finite cohomological dimension and the cohomology sheaves $H^i(K)$ are coherent for all $i$ \cite[I p.2]{SGA6}).
In fact, we can deduce from Proposition \ref{thm:wbc} that $\bR{f^{(m)}_{X/S}}_\ast E$ is locally of finite Tor-dimension \cite[V 3.5.9]{Berthelot:CCSCP}.
It therefore suffices to prove the pseudo-coherence of this complex.

First, in case $\fa_0=0$, that is, in case $S=S_0$,
this complex is quasi-isomorphic to a complex whose terms are finitely generated \cite[4.7]{LeStum-Quiros:EPLCCHL}
(as explained in Subsection \ref{ss:problem_on_freeness}, no problem occurs in the use of this result of Le Stum and Quir\'os 
because of the local freeness of $E$).
Hence the pseudo-coherence is obvious.

Next, assume that $(\fa_0,\fb_0,\gamma_0)=(\fa,\fb,\gamma)$.
In this case, the argument \cite[VII 1.1.1]{Berthelot:CCSCP} is used as follows.
For each natural number $n$, let $S_n$ be the closed subscheme defined by $\fa^{n+1}$.
Then, we have the exact sequence
\[
0\to \fa^n/\fa^{n+1}\to \sO_{S_n}\to \sO_{S_{n-1}}\to 0,
\]
which gives the distinguished triangle
\[
\begin{tikzpicture}
    \matrix (m) [matrix of math nodes, row sep=3em, column sep=2.5em, text height=2ex, text depth=0.5ex]
    { 
        & \bR{f_{X/S}^{(m)}}_\ast(E)\Lotimes \sO_{S_{n-1}} & \\
        \bR{f_{X/S}^{(m)}}_\ast(E)\Lotimes \fa^n/\fa^{n+1} & & \bR{f_{X/S}^{(m)}}_\ast(E)\Lotimes \sO_{S_n}. \\
    };
    \draw[->, snake=snake, segment amplitude=.4mm, segment length=2mm, line after snake=1mm]
        (m-1-2) -- (m-2-1);
    \path[->]
        (m-2-1) edge (m-2-3)
        (m-2-3) edge (m-1-2);
\end{tikzpicture}
\]
By Proposition \ref{thm:wbc}, the complex in the left-bottom is isomorphic to
\[
\bR{f_{X/S_0}^{(m)}}_\ast(E)\Lotimes_{\sO_{S_0}}\fa^n/\fa^{n+1},
\]
which is pseudo-coherent by the case where $\fa_0=0$.
Therefore \cite[I 2.5 b)]{SGA6}, noting that the ideal $\a_0$ is nilpotent,
the proof is obtained by induction.

At last, we consider the general case. Then, the category of the $m$-crystals in $\sO^{(m)}_{X/(S,\fa,\fb,\gamma)}$-modules
and that of the $m$-crystals in $\sO^{(m)}_{X/(S,\fa_0,\fb_0,\gamma_0)}$-modules are equivalent;
indeed, the question being local on $X$, we may assume that $X$ is embedded to a smooth scheme $Y$ over $S$
such that $Y\times_SS_0=X$, and then Proposition \ref{thm:crystal} applies.
Thus the theorem is reduced to the previous case by using Proposition \ref{thm:wbc}.
\end{proof}

\section{Other Applications.}

\subsection{K\"unneth formula.}

\begin{theorem}
    \label{thm:Kunneth}
    Let $(S,\fa,\fb,\gamma)$ be an $m$-PD scheme, and $R$ a quasi-compact $S$-scheme.
    Let $f\colon X\to R$ and $g\colon Y\to R$ be two quasi-compact, quasi-separated and smooth morphisms.
    Denote the fiber product $X\times_RY$ by $Z$, the projection $Z\to X$ \resp{$Z\to Y$} by $p$ \resp{by $q$}
    and the morphism $f\circ p=g\circ q$ by $h$;
\[
\begin{tikzpicture}
    \matrix (m) [matrix of math nodes, row sep=2em, column sep=3em, text height=2ex, text depth=0.5ex]
    { 
        Z & Y \\
        X & R & (S,\fa,\fb,\gamma) \\
    };
    \path[->,font=\scriptsize]
    (m-1-1) edge node[auto,swap] { $p$ } (m-2-1)
    (m-1-1) edge node[auto] { $q$ } (m-1-2)
    (m-2-1) edge node[auto,swap] { $f$ } (m-2-2)
    (m-1-2) edge node[auto] { $g$ } (m-2-2)
    (m-2-2) edge (m-2-3);
\end{tikzpicture}
\]
    If $E$ \resp{$F$} is a quasi-coherent and flat $\sO_{X/S}^{(m)}$-module \resp{$\sO_{Y/S}^{(m)}$-module},
    then there exists an isomorphism
    \[
    \begin{tikzpicture}
        \matrix (m) [matrix of math nodes,
                     row sep=3em, column sep=2.5em, text height=1.5ex, text depth=0.25ex]
                     { 
                         \bR{f_{\cris}^{(m)}}_\ast(E)\Lotimes_{\sO_{R/S}^{(m)}}\bR{g_{\cris}^{(m)}}_\ast(F) &
                         \bR{h_{\cris}^{(m)}}_\ast\big({p_{\cris}^{(m)}}^\ast(E)\otimes_{\sO_{Z/S}^{(m)}}{q_{\cris}^{(m)}}^\ast(F)\big). \\
                     };
        \path[->] (m-1-1) edge (m-1-2);
    \end{tikzpicture}
    \]
\end{theorem}

\begin{proof}
The construction of this morphism is straightforward by using the adjunction formula \cite[V 4.1.1]{Berthelot:CCSCP}.
In order to prove that the morphism is an isomorphim, the following proposition will suffice.
\end{proof}

\begin{proposition}
    Under the situation in  Theorem \ref{thm:Kunneth}, let $(\fa_0, \fb_0, \gamma_0)$ be a quasi-coherent $m$-PD sub-ideal of $\fa$,
    and assume that $R$ is the closed subscheme $S_0$ of $S$ defined by $\fa_0$.
    Then, the morphism
    \begin{equation}
        \label{eqn:Kunneth}
    \begin{tikzpicture}
        \matrix (m) [matrix of math nodes,
                     row sep=3em, column sep=2.5em, text height=1.5ex, text depth=0.25ex]
                     { 
                         \bR{f^{(m)}_{X/S}}_\ast(E)\Lotimes_{\sO_S}\bR{g^{(m)}_{X/S}}_\ast(F) &
                         \bR{h^{(m)}_{Z/S}}_\ast\left({p_{\cris}^{(m)}}^\ast(E)\otimes_{\sO_{Z/S}^{(m)}}{q_{\cris}^{(m)}}^\ast(F)\right) \\
                     };
        \path[->] (m-1-1) edge (m-1-2);
    \end{tikzpicture}
    \end{equation}
    is an isomorphism.
\end{proposition}

\begin{proof}
    We may assume that $X$ \resp{$Y$} is lifted to an affine smooth $S$-scheme $\bar{X}$ \resp{$\bar{Y}$} that has local coordinates.
    Then, $\bar{X}\times_S\bar{Y}$ is a lift of $Z$, which allows us to assume that $\bar{X}=X$, $\bar{Y}=Y$ and $Z=\bar{X}\times_S\bar{Y}$.
    We fix a system of local coordinates $\{t_1,\dots,t_n\}$ \resp{$\{t'_1,\dots,t'_{n'}\}$} of $X$ \resp{$Y$};
    the scheme $Z$ naturally has a system of local coordinates
    \[
    \{p^\ast(t_1),\dots,p^\ast(t_n), q^\ast(t'_1), \dots, q^\ast(t'_{n'})\},
    \]
    which gives us a natural isomorphism
    \[
    \begin{tikzpicture}
        \matrix (m) [matrix of math nodes,
                     row sep=3em, column sep=2.5em, text height=1.5ex, text depth=0.25ex]
                     { 
                     p^\ast(\hO_{X/S}^{(m),\bullet})\otimes q^\ast(\hO_{Y/S}^{(m),\bullet}) &
                     \hO_{Z/S}^{(m),\bullet}. \\
                     };
        \path[->] (m-1-1) edge node[above,inner sep=0.5pt] {$\sim$} (m-1-2);
    \end{tikzpicture}
    \]
    
    The proposition, therefore, follows from the commutative diagram
    \[
    \begin{tikzpicture}
        \matrix (m) [matrix of math nodes,
                     row sep=2em, column sep=0em, text height=1.5ex, text depth=1ex]
                     { 
                     \left(\bR{f_{X/S}^{(m)}}_\ast(E)\Lotimes_{\sO_S}\bR{g_{Y/S}^{(m)}}_\ast(F)\right)^{\oplus p^{m(n+n')}} & \\
                      & 
                     \bR{h_{Z/S}^{(m)}}_\ast\left({p_{\cris}^{(m)}}^\ast(E)\otimes_{\sO_{Z/S}^{(m)}}{q_{\cris}^{(m)}}^\ast(F)\right)^{\oplus p^{m(n+n')}}\\
                     \bR{f_{X/S}^{(m)}}_\ast(E)^{\oplus p^{mn}}\Lotimes_{\sO_S}\bR{g_{Y/S}^{(m)}}_\ast(F)^{\oplus p^{mn'}} & \\
                     &  h_*(p^\ast(E_X)\otimes_{\sO_Z}q^\ast(F_Y)\otimes_{\sO_Z}\hO_{Z/S}^{(m),\bullet}); \\
                     f_\ast(E_X\otimes_{\sO_X}\hO_{X/S}^{(m),\bullet})\otimes_{\sO_S} g_\ast(F_Y\otimes_{\sO_Y}\hO_{Y/S}^{(m),\bullet}) &  \\
                     };
        \draw[->] (m-1-1) edge node[below,inner sep=2pt,sloped] {$\sim$} (m-3-1);
        \path[->] (m-1-1) edge (m-2-2.north);
        \path[->] (m-2-2) edge node[above,swap,inner sep=2pt,sloped] {$\sim$} (m-4-2);
        \path[->] (m-3-1) edge node[below,inner sep=2pt,sloped] {$\sim$} (m-5-1);
        \path[->] (m-5-1) edge node[above,inner sep=2pt,sloped] {$\sim$} (m-4-2.south);
    \end{tikzpicture}
    \]
    here, the morphism from the top-left toward bottom-right is the direct sum of $p^{m(n+n')}$ copies of (\ref{eqn:Kunneth}),
    and that from the top-left toward bottom is induced from
    \begin{eqnarray*}
        \left({f_{\cris}^{(m)}}_\ast(E)\otimes{g_{\cris}^{(m)}}_\ast(F)\right)^{\oplus p^{m(n+n')}} 
        & \to & {f_{\cris}^{(m)}}_\ast(E)^{\oplus p^{mn}}\otimes {g_{\cris}^{(m)}}_\ast(F)^{\oplus p^{mn'}}\\
        e_{Z,(I,I')} & \mapsto & e_{X,I} \otimes e_{Y,I'}.
    \end{eqnarray*}
    where $I\in \fB^{(m)}_n$, $I'\in\fB^{(m)}_{n'}$ and $e_{X,I}$ \resp{$e_{Y,I'}$, $e_{Z, (I,I')}$} is the canonical basis of each direct sum.
\end{proof}

%

\subsection{Frobenius Descent.}
\label{ss:frobenius_descent}

In this subsection, we correct the proof \cite[5]{LeStum-Quiros:EPLCCHL}  of Frobenius descent.

\begin{lemma}
    \label{lem:spencer}
    Let $(S,\fa,\fb,\gamma)$ be an $m$-PD scheme,
    let $X$ be an $n$-dimensional smooth scheme over $S$ that has local coordinates,
    and let $\sF$ be a $\sD_{X/S}^{(m)}$-module.
    Then, there exists a natural isomorphism
    \begin{equation}
    \begin{tikzpicture}
        \matrix (m) [matrix of math nodes,
                     row sep=3em, column sep=2.5em, text height=1.5ex, text depth=0.25ex]
                     { 
                     \sF\otimes_{\sO_X}\hO_{X/S}^{(m),\bullet} &
                     R\sHom_{\sD_{X/S}^{(m)}}\big(\sO_X, \sF\big)^{\oplus p^{mn}}. \\
                     };
        \path[->] (m-1-1) edge node[above,inner sep=0.5pt] {$\sim$} (m-1-2);
    \end{tikzpicture}
    \label{eqn:descent}
    \end{equation}
    \label{lem:descent}
\end{lemma}

\begin{proof}
    Because each term of the higher de Rham complex is free, the complex
    \[
    \sHom_{\sO_X}\big(\hO^{(m),\bullet}_{X/S}, \sD_{X/S}^{(m)}\big)
    \]
    gives a free resolution of $\sO_X^{\oplus p^{mn}}$ as $\sD_{X/S}^{(m)}$-modules.
    Because of the freeness again, the morphism
    \[
    \sF\otimes_{\sO_X}\hO^{(m),\bullet}_{X/S}\to\sHom_{\sD_{X/S}^{(m)}}\left(\sHom_{\sO_X}\big(\hO^{(m),\bullet}_{X/S},\sD_{X/S}^{(m)}\big),\sF\right),
    \]
    is an isomorphism, which completes the proof.
\end{proof}

In the remainder of this subsection, let $(S,\fa,\fb,\gamma)$ be an $m$-PD scheme,
let $(\fa_0,\fb_0,\gamma_0)$ a quasi-coherent $m$-PD sub-ideal of $\fa$ that satisfies $p\in\fa_0$,
and denote by $S_0$ the closed subscheme defined by $\fa_0$.
Let $f\colon X_0\to S_0$ be a smooth morphism such that
the $m$-PD structure $(\fb_0,\gamma_0)$ extends to $\sO_{X_0}$,
and let $F_0\colon X_0\to X_0'$ denote the $s$-th iterate of the relative Frobenius
of $f$ for a natural number $s$.
We denote by $f'\colon X_0'\to S_0$ the natural morphism.

Under this situation, the morphism $F_0$ induces
a morphism of topoi \cite[(4.1)]{Berthelot:LI}
\[
F_0\colon (X_0/S)_{\cris}^{(m+s)}\to (X'_0/S)_{\cris}^{(m)}.
\]
In order to recall the construction of the inverse image functor of $F_0$, 
we fix an object $(U,T,J,\delta)$ of $\Cris^{(m+s)}(X_0/S)$
and set some notation.
Let $T_0$ denote the closed subscheme of $T$ defined by $J+p\sO_T$, and
$T_1$ the closed subscheme of $T_0$ defined by $I^{(p^m)}\sO_{T_0}$,
where $I$ denotes the ideal of $\sO_T$ defined by the closed immersion
$U\hookrightarrow T$.
We denote the $s$-th iterate of the relative Frobenius morphism of $U$ \resp{$T_1$, $T_0$} by
$F_{0,U}\colon U\to U'$ \resp{$F_{0,T_1}\colon T_1\to T'_1$, $F_{0,T_0}\colon T_0\to T'_0$}.
Now, we have a continuous functor
\[
\Cris^{(0)}(T'_1/S)\to\Cris^{(m)}(X'/S)
\]
defined as follows.
For an object $(W, V, \delta_0)$ in $\Cris^{(0)}(T'_1/S)$,
if $I_0$ denotes the ideal of $\sO_V$ defined by the closed immersion $W\hookrightarrow V$,
the image of $(W, V, \delta_0)$ through the functor above is
$(U'\cap W, V, I_0, \delta_0)$;
in fact, $(I_0, \delta_0)$ is the $m$-PD structure of the ideal of
the closed immersion $U'\cap W\hookrightarrow W\hookrightarrow V$
since, because $(I^{(p^m)}\sO_{T_0})^{(p^s)}=0$, we have $T'_0=T'_1$.
We may directly prove that this defines a continuous functor.

Now, for a sheaf $E$ on $\Cris^{(m)}(X'_0/S)$,
the section of $F_0^\ast E$ on $(U,T,J,\delta)$ is described as follows.
If $E_1$ denotes the sheaf on $\Cris^{(0)}(T'_1/S)$ induced by $E$,
and if $E_2$ denotes the image of $E_1$ 
through the inverse image functor of $(F_{0,T_0})^{(0)}_{\cris}$,
we have
\[
(F_0^\ast E)(U,T,J,\delta)=E_2(T_0,T).
\]

If $E$ is an $m$-crystal in $\sO_{X'_0/S}^{(m)}$-modules,
then $F_0^\ast E$ is an $(m+s)$-crystal in $\sO_{X'/S}^{(m+s)}$-modules;
in fact, the question being local, we may assume that
$X_0$ \resp{$X'_0$} can be lifted into a smooth $S$-scheme $X$ \resp{$X'$}
and $F_0$ can be lifted into a morphism $F\colon X\to X'$,
in which case $F_0^\ast E$ is an $(m+s)$-crystal \cite[2.2.3]{Berthelot:DMAII}.

\begin{proposition}
    Let $E'$ be a locally free $\sO_{X_0'/S}^{(m)}$-module of finite rank, and
    denote by $E$ the $\sO_{X_0/S}^{(m+s)}$-module $F_0^\ast E'$,
    which is an $(m+s)$-crystal as noted above.

    Then, the morphism
    \begin{equation}
    F_0^*\colon R{u^{(m)}_{X_0'/S}}_\ast(E')\to R{u_{X_0/S}^{(m+s)}}_\ast(E)
    \label{eqn:frob}
\end{equation}
    is an isomorphism.
    \label{prop:descent}
\end{proposition}

Before starting the proof, we show the following lemma.

\begin{lemma}
    Assume that $p\sO_S=0$, that $\fa_0=0$ and let us omit the subscripts ``0''
    in the notation.
    Moreover, assume that $X$ and $X'$ has a system of local coordinates;
    we fix one $\{t_1,\dots,t_n\}$ of $X$ and $\{t_1',\dots,t_n'\}$ of $X'$.

    Then, the morphism $F\colon X'\to X$ induces a quasi-isomorphism
    \[
    \varphi\colon E'_{X'}\otimes\big(\hO_{X'/S}^{(m),\bullet}\big)^{\oplus p^{sn}}\to
    E_X\otimes \hO_{X/S}^{(m+s),\bullet}.
    \]
\end{lemma}

\begin{proof}
Since Berthelot \cite[2.2.2 (i)]{Berthelot:DMAII} shows that $F$ induces a morphism of simplicial schemes
\[
F^\ast\colon P_{X'/S}^{(m)}(\bullet)\to P_{X/S}^{(m+s)}(\bullet),
\]
we get a morphism
\begin{equation}
F^\ast\colon E'_{X'}\otimes\sP_{X'/S}^{(m),\bullet}\to E_X\otimes\sP_{X/S}^{(m+s),\bullet}.
\label{eqn:frob_P}
\end{equation}
We show that this induces a morphism
\[
F^\ast\colon E'_{X'}\otimes\hO_{X'/S}^{(m),\bullet}\to E_X\otimes\hO_{X/S}^{(m+s),\bullet}.
\]
Put $\tau_i=d^\ast(t_i)$ and $\tau_i'=d^\ast(t_i')$,
and recall that \cite[2.2.4 (i)]{Berthelot:DMAII}, for each $I\in \bN^n\setminus\{0\}$,
this morphism takes $\tau'^{\{I\}}$ to $\tau^{\{p^sI\}}$.
Therefore, (\ref{eqn:frob_P}) induces
\[
E'_{X'}\otimes N\sP^{(m),\bullet}_{X'/S}\to E_X\otimes N\sP_{X/S}^{(m),\bullet}.
\]
Moreover, if $I$ is not equal to $p^m\one_i$ for any $i=1,\dots,n$,
then $p^sI$ is not of the form $p^{m+s}\one_j$,
which shows that the image of $\sK_{X'/S}^{(m)}$ under $F^\ast$
lies in $\sK_{X/S}^{(m+s)}$.

Now, for each $J\in\fB^{(m)}_s$, consider the morphism
\[
F_J\colon E'_{X'}\otimes\hO_{X'/S}^{(m),\bullet}\to E_X\otimes\hO_{X/S}^{(m+s),\bullet}
\]
obtained by multiplying $\tau^{J}$ from the left after $F^\ast$ above.
This morphism is actually a zero morphism if $J\neq 0$ since
$\tau^J\tau_i^{p^{m+s}}=0$ in $\hO_{X/S}^{(m+s),1}$.
We define the morphism $\varphi$ to be the direct sum of $F_J$'s
for $J\in\fB^{(m)}_s$.

The morphism $\varphi$ fits into the commutative diagram
    \[
    \begin{tikzpicture}
        \matrix (m) [matrix of math nodes,
                     row sep=3em, column sep=2.5em, text height=1.5ex, text depth=0.25ex]
                     { 
                     E'_{X'}\otimes\big(\hO_{X'/S}^{(m),\bullet}\big)^{\oplus p^{sn}} &
                     E_X\otimes\hO_{X/S}^{(m+s),\bullet} \\
                     R\sHom_{\sD_{X'/S}^{(m)}}(\sO_{X'}, E')^{\oplus p^{(m+s)n}} &
                     R\sHom_{\sD_{X/S}^{(m+s)}}(\sO_X, E)^{\oplus p^{(m+s)n}}, \\
                     };
        \path[->,font=\scriptsize] (m-1-1) edge node[auto,inner sep=0.5pt] {$\varphi$} (m-1-2)
        (m-1-1) edge node[auto,inner sep=0.5pt]{$\varphi$} (m-1-2)
        (m-1-1) edge node[auto,swap] {(\ref{eqn:descent})${}^{\oplus p^{sn}}$} (m-2-1)
        (m-1-2) edge node[auto] {(\ref{eqn:descent})} (m-2-2)
        (m-2-1) edge node[above,inner sep=2pt] {$\sim$} (m-2-2);
    \end{tikzpicture}
    \]
    where the morphism in the bottom is the direct sum of $p^{(m+s)n}$ copies of
    \[
    F^\ast\colon R\sHom_{\sD_{X'/S}^{(m)}}(\sO_{X'}, E_{X'}) \to R\sHom_{\sD_{X/S}^{(m+s)}}(\sO_X, E_X);
    \]
    this is a quasi-isomorphism by Berthelot \cite[2.3.8 (i)]{Berthelot:DMAII}, which shows the lemma.
\end{proof}

\begin{proof}[Proof of Proposition \ref{prop:descent}]
    First, we reduce the proof to the case where the assumptions
    in the previous lemma holds.
    The question being local, we may assume that
    $X_0$ \resp{$X'_0$} can be lifted into a smooth $S$-scheme $X$ \resp{$X'$}
    and $F_0$ can be lifted into a morphism $F\colon X\to X'$,
    in which case Corollary \ref{cor:cryst_equiv} and Proposition \ref{thm:wbc} allows us
    to assume that $(\fa_0, \fb_0, \gamma_0)=(\fa,\fb,\gamma)$.
    Then, by using the distinguished triangle as in the proof
    of Theorem \ref{thm:finiteness},
    we may assume that $p\sO_S=0$ and $\fa =0$.
    Finally, again since the question is local, we may argue under the
    situation of the previous lemma.

Now, the construction shows that the following diagram commutes:
    \[
    \begin{tikzpicture}
        \matrix (m) [matrix of math nodes,
                     row sep=3em, column sep=2.5em, text height=1.5ex, text depth=0.25ex]
                     { 
                     R{u_{X'/S}^{(m)}}_\ast(E')^{\oplus p^{(m+s)n}} &
                     R{u_{X/S}^{(m+s)}}_\ast(E)^{\oplus p^{(m+s)n}} \\
                     E'_{X'}\otimes\big(\hO_{X'/S}^{(m),\bullet}\big)^{\oplus p^{sn}} &
                     E_X\otimes\hO_{X/S}^{(m+s),\bullet}; \\
                     };
        \path[->,font=\scriptsize]
        (m-1-1) edge (m-1-2)
        (m-2-1) edge node[auto,inner sep=0.5pt] {$\varphi$} (m-2-2)
        (m-1-1) edge node[above,inner sep=2pt,sloped] {$\sim$} (m-2-1)
        (m-1-2) edge node[above,inner sep=2pt,sloped] {$\sim$} (m-2-2);
    \end{tikzpicture}
    \]
    here, the morphism in the top is the direct sum of $p^{(m+s)n}$ copies of the morphism
    (\ref{eqn:frob}).
    This shows the proposition.
\end{proof}

\bibliographystyle{dagaz}
\bibliography{math}

\end{document}